\documentclass[a4paper,reqno, 11pt]{amsart}
\usepackage{amsmath,amssymb,amsthm,enumerate}

\usepackage{paralist}
\usepackage{graphics}
\usepackage{epsfig}
\usepackage{amsfonts}
\usepackage{amssymb}
\usepackage{color}

\usepackage{mathtools}

\usepackage[numbers,sort&compress]{natbib}

\usepackage{mathrsfs}
\usepackage{lineno}

\usepackage[svgnames]{xcolor}

\usepackage{pifont}
\usepackage{mathtools}

\usepackage[misc]{ifsym}

\usepackage[margin=1in]{geometry}

\def\ifl{\iffalse }

\def\bc{\begin{center}} \def\ec{\end{center}}
\def\ba{\begin{array}} \def\ea{\end{array}}
\def\bea{\begin{eqnarray}} \def\eea{\end{eqnarray}}
\def\beaa{\begin{eqnarray*}} \def\eeaa{\end{eqnarray*}}

\newtheorem{thm}{Theorem}[section]

\newtheorem{lem}{Lemma}[section]
\theoremstyle{remark}
\newtheorem{rem}{Remark}[section]
\newtheorem*{rem*}{Remark}
\newtheorem{cor}{Corollary}[section]

\numberwithin{equation}{section}

\title{On fractional smoothness of modulus of functions}

\author[D. Li]{ Dong Li}
\address{D. Li, Department of Mathematics, the Hong Kong University of Science \& Technology, Clear Water Bay, Kowloon, Hong Kong}
\email{madli@ust.hk}

\begin{document}

\begin{abstract}
We consider the Nemytskii operators $u\to |u|$ and $u\to u^{\pm}$ in a bounded domain
$\Omega$ with $C^2$ boundary. We give elementary proofs of the boundedness  in 
$H^s(\Omega)$ with $0\le s<3/2$.
\end{abstract}

\maketitle

\section{introduction}
Let $\Omega$ be a nonempty open bounded set in $\mathbb R^d$. 
For $0<\gamma<1$ and $f \in C^1(\Omega)$, define the nonlocal $\dot H^{\gamma}$
semi-norm as
\begin{align}
\| f \|_{\dot H^{\gamma}(\Omega)}^2
= \int_{x, y \in \Omega, x \ne y}
\frac { |f(x) - f(y)|^2} {|x-y|^{d+2\gamma} } dx dy.
\end{align}
For $f \in C^2(\Omega)$ and $0<\gamma<1$, we define 
\begin{align}
\| f \|_{H^{1+\gamma} (\Omega)} =
\| f \|_{H^1(\Omega)} + \| \partial f   \|_{\dot H^{\gamma} (\Omega)},
\end{align}
where (and throughout this note) $\partial f = (\partial_{x_1} f, \cdots, \partial_{x_n} f )$
denotes the usual gradient.  Throughout this note we shall only be concerned with real-valued
functions, however with some additional work the results can be generalized to complex-valued functions.  Define the Nemytskii operators 
\begin{align}
T_1 u =|u|, \quad T_2 u = u^+ = \max\{u,0\}, \quad T_3 u = u^-=\max\{-u,0\}.
\end{align}
The purpose of this note is to give an elementary proof of  the following.

\begin{thm}[Boundedness  in $H^{\frac 32-}(\Omega)$] \label{thmMa}
Let $d\ge 1$ and $0\le s <\frac 32$.  Assume $\Omega$ is a nonempty open bounded set in $\mathbb R^d$
with $C^2$ boundary, i.e.   locally it can be written as the graph of a $C^2$ function on
$\mathbb R^{d-1}$. Then $T_i$, $i=1,2,3$ are bounded  on $H^s(\Omega)$. More 
precisely,
\begin{align}
\sum_{i=1}^3 \| T_i u \|_{H^s (\Omega)} \le \alpha_1 \| u \|_{H^s(\Omega)},
\end{align}
where $\alpha_1>0$ depends on ($s$, $\Omega$, $d$).
\end{thm}
\begin{rem}
In Theorem \ref{thmMa}, the case $0<s<1$ is trivial thanks to the simple inequality
$||x|-|y||\le |x-y|$ for any $x, y \in \mathbb R$. The case $s=1$ corresponds to
the well-known distributional calculation $\partial (|u|) =  \mathrm{sgn}(u)
\partial u $ for $u\in H^1$.  Thus only the case $1<s<\frac 32$ requires some work.
The obstruction $s<\frac 32$ is clear since there are jump discontinuities of the gradient
along manifolds of dimension $d-1$. 
In 1D  one can take a smooth compactly supported
 function $\phi$ such that $\phi(x) \equiv x$ for $x$ near the origin. It is trivial
 to verify that $|\phi| \notin H^{\frac 32}$. 
\end{rem}
\begin{rem}
It follows from our proof that for $1<s<\frac 32$,
 $T_i$ maps bounded sets in $H^s(\Omega)$ to pre-compact sets in $H^1(\Omega)$. 
 This fact has important applications in the convergence of approximating solutions
 in some nonlinear PDE problems. 
 \end{rem}
There is by now an enormous body of literature on extension, composition, regularity and stability of nonlocal operators and we shall not give a survey on the state of art in this short note.
For $C^{\infty}$ boundary $\partial \Omega$, one can use interpolation to define
the fractional spaces $H^s(\Omega)$ which can be regarded as restrictions of functions in
$H^s(\mathbb R^n)$ (cf. \cite{LM72}). In \cite{Mey91} Bourdaud and Meyer proved the boundedness
of  $T_1$  in Besov spaces $B^s_{p,p}(\mathbb R^d)$, $0<s<1+\frac 1p$, $1\le p\le \infty$.
By using linear spline approximation theory, Oswald \cite{Os92} showed that $T_1$ is bounded in $B^s_{p,q}(\mathbb R)$, $1\le p,q\le \infty$ if and only if $0<s<1+\frac 1p$.
In \cite{S95}, Savar\'e showed the regularity of $T_2$ in the space 
$BH(\Omega)=\{u \in W^{1,1}(\Omega): \, \text{$D^2 u$ is a matrix-valued 
bounded measure}\}$, i.e. $\nabla u \in \mathrm{BV}(\Omega;\mathbb R^d)$. 
One should note that in general $T_1$ increases the $H^s$ norm for $1<s<\frac 32$. For example,
for $1<s<\frac 32$,  R. Musina and A.I. Nazarov \cite{MN17} showed that if $u
\in H^s(\mathbb R^d)$ changes
sign, then
\begin{align}
\langle (-\Delta_{\mathbb R^d} )^s |u|, |u| \rangle > 
\langle (-\Delta_{\mathbb R^d} )^s u, u\rangle.
\end{align}
We refer to \cite{BS11,MN17} and the references therein for a more detailed survey of composition
operators in function spaces with various fractional order of smoothness.

The rest of this note is organized as follows. In Section 2 we give the proof for the one dimensional
case. In Section 3 we give the details for general dimensions $d\ge 2$.

\subsection*{Notation}
For any two quantities $A$, $B\ge 0$, we write  $A\lesssim B$ if $A\le CB$ for some unimportant constant $C>0$. 
We write $A\sim B$ if $A\lesssim B$ and $B\lesssim A$. 

\subsection*{Acknowledgement.}
D. Li is supported in part by Hong Kong RGC grant GRF 16307317 and 16309518.
The author is indebted to Xiaoming Wang for raising this intriguing question.

\section{The one dimensional case}

\begin{lem} \label{lem_H1}
Let $0<\alpha<1$ and $0<\delta<1-\alpha$. Consider
\begin{align}
F(k) = \int_0^{\infty} 
\frac {1-y^{-\delta}} {(k+|1-y|^2)^{\frac{1+\alpha}2} } dy, \quad k>0.
\end{align}
Then $F(k)$ is uniformly bounded and 
\begin{align}
\lim_{k\to 0} F(k) =F(0) >0.
\end{align}
\end{lem}
\begin{proof}
By using Lebesgue Dominated Convergence we have $\lim_{k\to 0}F(k)=F(0)$.  
Note that 
\begin{align}
F(0) = \int_0^1 \frac{ (1-y^{-\delta} ) (1-y^{\alpha+\delta-1})} {|1-y|^{1+\alpha}} dy>0.
\end{align}

\end{proof}

\begin{lem} \label{lem_H2}
Let $0<\alpha<1$. Assume $u$ is bounded on $[0,\infty)$ and $|u(x)| \lesssim x^{-2}$ for
$x\ge 1$. Then
\begin{align}
\int_0^{\infty} \int_0^{\infty}
\frac{ (u(x)-u(y))^2} {|x-y|^{1+\alpha}} dx dy
\gtrsim \int_0^{\infty} \frac{u(x)^2}{x^{\alpha}} dx.
\end{align}
\end{lem}

\begin{proof}
We use by now the standard super-harmonic approach (cf. \cite{BD11, LRNew}).
Observe for $w>0$, 
\begin{align}
(u(x)-u(y))^2 \ge u^2(x) \frac{w(x)-w(y)}{w(x)} + u^2(y) \frac{w(y)-w(x)}{w(y)}.
\end{align}
For $x>0$, take $w(x) = x^{-\delta}$ with $0<\delta<1-\alpha$. By Lemma \ref{lem_H1},
it is not difficult to check that
\begin{align}
\sup_{x>0, \epsilon>0} \frac {x^{\alpha}} {w(x)}
\Bigl| \int_0^{\infty} \frac {w(x) -w(y)} { (\epsilon^2 + |x-y|^2)^{\frac {1+\alpha}2} }
dy
\Bigr| \lesssim 1.
\end{align}
Thus
\begin{align}
\int \frac{(u(x)-u(y))^2}{(\epsilon^2 +|x-y|^2)^{\frac{1+\alpha}2} }
dx dy \ge 2 \int u(x)^2 \frac 1{w(x)} \int \frac {w(x)-w(y)} {
(\epsilon^2+|x-y|^2)^{\frac {1+\alpha}2} } dy dx.
\end{align}
Sending $\epsilon \to 0$ then yields the result.
\end{proof}

\begin{lem} \label{lem_H3}
Let $0<\alpha<1$. Assume $u$ is bounded on $[0,1]$. 
Suppose $\int_0^1 u(x) dx =0$. Then 
\begin{enumerate}
\item $\int_0^1 \int_0^1 (u(x) -u(y))^2 dx dy = 2 \| u\|_2^2 $. 
\item $\int_0^1 \int_0^1 \frac {(u(x) -u(y))^2}{
|x-y|^{1+\alpha} } dx dy
\gtrsim \int_0^1 \frac {u(x)^2} { \min\{ x^{\alpha},  (1-x)^{\alpha}\} } dx.$
\end{enumerate}
More generally, there are constants $\beta_1>0$, $\beta_2>0$ depending only on $\alpha$, such that
for any finite interval on $\mathbb R$, it holds that (below $|I|$ denote the length
of the interval)
\begin{align} \label{ineqI}
\int_{x\ne y \in I}
\frac {|f(x)-f(y)|^2}{|x-y|^{1+\alpha}} dx dy
+ \beta_1 |I|^{-\alpha} \| f\|_{L^2(I)}^2
\ge \beta_2 \int_{I}
\frac {f(x)^2}{\mathrm{dist}(x, I^c)^{\alpha} } dx.
\end{align}
If $\int_I f  dx =0$, then we have
\begin{align}
\int_{x\ne y \in I}
\frac {|f(x)-f(y)|^2}{|x-y|^{1+\alpha}} dx dy
\ge \beta_3 \int_{I}
\frac {f(x)^2}{\mathrm{dist}(x, I^c)^{\alpha} } dx, \notag 
\end{align}
where $\beta_3>0$ depends only on $\alpha$.
\end{lem}
\begin{rem}
Remarkably, the Hardy inequality (2) does not hold if we assume $u \in C_c^{\infty}((0,1))$.
As a counterexample, one can take $u_n \in C_c^{\infty} ((0,1))$  such
that $u_n(x)=1$ for $\frac 1n \le x \le 1- \frac 1n$, $u_n(x)=0$ for $x\le \frac 1 {2n}$ and
$1-\frac 1{2n} \le x\le 1$, and $|u^{\prime}| \lesssim n$.  Then we have
\begin{align}
&\int_0^1 \frac {u_n(x)^2} { x^{\alpha} (1-x)^{\alpha}} dx \sim 1; \\
& \int_0^1 \int_0^1 \frac { (u_n(x) -u_n(y) )^2} {|x-y|^{1+\alpha} }
dx dy \lesssim n^{-(1-\alpha)}.
\end{align}
One can see \cite{Dyda04} for an extensive discussion on general fractional order Hardy inequalities and counterexamples. 
\end{rem}

\begin{proof}
The first identity is obvious. For the second inequality, one can apply Lemma \ref{lem_H2}
to $u(x) \chi(x)$ where $\chi$ is a smooth cut-off function supported in $[0, 2/3]$. This then
yields 
\begin{align}
\int_0^1 \frac {|u(x)-u(y)|^2}{|x-y|^{1+\alpha}} dx dy + \operatorname{const}
\cdot \int_0^1 u(x)^2 dx
\gtrsim \int_0^{\frac 12} \frac {u(x)^2}{x^{\alpha}} dx.
\end{align}
The desired inequality  follows easily by using the first identity. To get the extra factor
$(1-x)^{-\alpha}$ one can invoke the symmetry $x\to 1-x$. 
The inequality \eqref{ineqI} follows from rescaling and reducing to the case $I=(0,1)$.
\end{proof}

The following theorem is a special case of Bourdaud-Meyer \cite{Mey91}. We reproduce the
proof here to highlight the needed changes for the finite domain case (see Theorem 
\ref{thm2Int}).
\begin{thm}
Let $0<s<1/2$. We have
\begin{align}
&\| u^{\prime} \cdot 1_{u>0} \|_{H^{s} (\mathbb R)}
\lesssim \| u^{\prime} \|_{H^{s} (\mathbb R)}; \\
& \| T_2 u \|_{H^{1+s}(\mathbb R)} \lesssim \| u \|_{H^{1+s}(\mathbb R)}.
\end{align}
\end{thm}
\begin{proof}
Set $\alpha=2s$. Write $I=\{x:\, u(x)>0 \}$ as a countable disjoint union of intervals $I_j$ such
that each $I_j=(a_j, b_j)$ satisfies $u(a_j)=u(b_j)=0$.  Here if $a_j$ or $b_j$ are infinity the value
of $u(a_j)$ or $u(b_j)$ are understood in the limit sense. 
By using Lemma \ref{lem_H3}, we have
\begin{align}
\| u^{\prime} 1_{u>0} \|_{H^s(\mathbb R)}^2
&\lesssim \| u^{\prime} \|_{H^s(\mathbb R)}^2
+ \int_{x\in\mathbb R:\, u(x)>0} (u^{\prime}(x))^2 \int_{u(y)<0}
\frac 1 {|x-y|^{1+\alpha}} dy dx \notag \\
& \lesssim \| u^{\prime}\|_{H^{s}(\mathbb R)}^2 + 
\sum_j \int_{I_j} \frac { (u^{\prime}(x))^2} {  (\mathrm{dist} (x, I_j^c) )^{\alpha} } dx
\notag \\
& \lesssim \| u^{\prime}\|_{H^{s}(\mathbb R)}^2
+ \sum_j \int_{I_j \times I_j} \frac { (u^{\prime}(x) - u^{\prime}(y) )^2}
{ |x-y|^{1+\alpha} } dx d y  \notag \\
& \lesssim \| u^{\prime}\|_{H^{s}(\mathbb R)}^2.
\end{align}
\end{proof}

\begin{cor} \label{cor2.0A}
Let $0<s<1/2$ and $d\ge 2$. Then
\begin{align}
\| T_2 u \|_{H^{1+s}(\mathbb R^d)} \lesssim \| u \|_{H^{1+s}(\mathbb R^d)}.
\end{align}
\end{cor}
\begin{proof}
By a partition of unity in frequency space, we have
\begin{align}
\| T_2 u \|_{H^{1+s}(\mathbb R^d)} \lesssim
\sum_{j=1}^d \| T_2 u \|_{L_{x_j^{\prime}}^2 H_{x_j}^{1+s} (\mathbb R^{d-1}\times
\mathbb R)},
\end{align}
where $x_j^{\prime}$ denotes the variables $x_l$ with $l\ne j$. 
The desired result then follows from the one-dimensional case.
\end{proof}

\begin{thm}[Boundedness of $T_2$ on a finite interval] \label{thm2Int}
Let $0<s<1/2$ and $\Omega=(0,1)$. We have
\begin{align}
&\| u^{\prime} \cdot 1_{u>0} \|_{H^{s} (\Omega)}
\lesssim \| u^{\prime} \|_{H^{s} (\Omega)}; \\
& \| T_2 u \|_{H^{1+s}(\Omega)} \lesssim \| u \|_{H^{1+s}(\Omega)}.
\end{align}
\end{thm}
\begin{proof}
Set $\alpha=2s$.  If $\{u>0\}=(0,1)$ there is nothing to prove.
Otherwise write $\{x\in (0,1): \, u(x)>0 \}$ as a countable union 
of  maximal intervals $I_j$ such that for each $I_j=(a_j, b_j)$, either
$u(a_j)=u(b_j)=0$ with $0<a_j<b_j<1$, or $a_j=0$, $u(b_j)=0$, or $u(a_j)=0$, $b_j=1$.
In yet other words, either $I_j$ is strictly contained in $\Omega$, or $I_j$ intersects with the
boundary of $\Omega$.  One should note that those ``interior" intervals can be treated 
in the same
way as before, and we only need to take care of those $I_j$ which intersects with the boundary
of $\Omega$.
With no loss we consider the case  $I_1=(0, b)$ where $u(0)>0$,
$u(b)=0$. We discuss several further cases. 

Case 1: $b\ge 1/4$. Clearly then
\begin{align}
\int_{I_1} \frac {(u^{\prime}(x))^2} {\mathrm{dist}(x, I_1^c)^{\alpha}}
dx \lesssim \int_{I_1\times I_1} 
\frac{ (u^{\prime}(x)-u^{\prime}(y))^2} {|x-y|^{1+\alpha}} dx dy
+ |I_1|^{-\alpha} \|u^{\prime}\|_{L^2(I_1)}^2  \lesssim
\operatorname{OK}.
\end{align}

Case 2: $0<b<1/4$. We set $J_1=(b, 1)$. Observe that $|J_1|\ge 3/4$ and
\begin{align}
\int_{J_1} \frac { (u^{\prime}(y))^2} { \mathrm{dist} (y, J_1^c)^{\alpha}}
dy
\lesssim \int_{J_1\times J_1}
\frac { (u^{\prime}(y_1) -u^{\prime}(y_2) )^2}  {|y_1-y_2|^{1+\alpha}}
dy_1 dy_2+  |J_1|^{-\alpha} \| u^{\prime} \|_{L^2(J_1)}^2.
\end{align}
Clearly
\begin{align}
\int_{I_1} \frac{(u^{\prime}(x) )^2} {\mathrm{dist}(x,I_1^c)^{\alpha} }
& \lesssim \int_{I_1 \times J_1}
\frac{(u^{\prime}(x))^2} {|x-y|^{1+\alpha}} dx dy \notag \\
& \lesssim \int_{I_1\times J_1}
\frac{ (u^{\prime}(x)-u^{\prime}(y) )^2} {|x-y|^{1+\alpha}} dxdy
+\int_{J_1}
\frac{ (u^{\prime}(y))^2} {\mathrm{dist}(y, J_1^c)^{\alpha}} dy.
\end{align}
Thus this case is also OK.

\end{proof}

\section{The general case $d\ge 2$}

\begin{lem} \label{lem2.1A}
Let $0<s<1/2$. Denote $Q_+=\{ (y_1, y_2):\, |y_1|<1, 0\le y_2 <f(y_1)\}$ where $f$
is a positive continuous function.
Suppose $K$ is a nonempty compact set which is  properly contained in $Q_+$.
In yet other words, the set $\tilde K = \{ (y_1, \pm y_2): \, (y_1,y_2) \in K \}$ is in
the interior of $ Q=\{(y_1, y_2):\, |y_1|<1, |y_2|<f(y_1)\}$.  Assume $g$ vanishes on
$Q_+ \setminus K$. 
We have
\begin{align} \label{eq2.1}
\int_{x \in Q_+, y \in Q_+}
 \frac { |g(y)|^2} { (|x_1-y_1|+|x_2+y_2|)^{2+2s} }
 dy dx \le C_1 \| g \|_{H^s(Q_+)}^2,
 \end{align}
 where $C_1>0$ depends only on ($s$, $K$, $f$). 
\end{lem}
\begin{figure}[!h]
\centering
\includegraphics[width=0.4\textwidth]{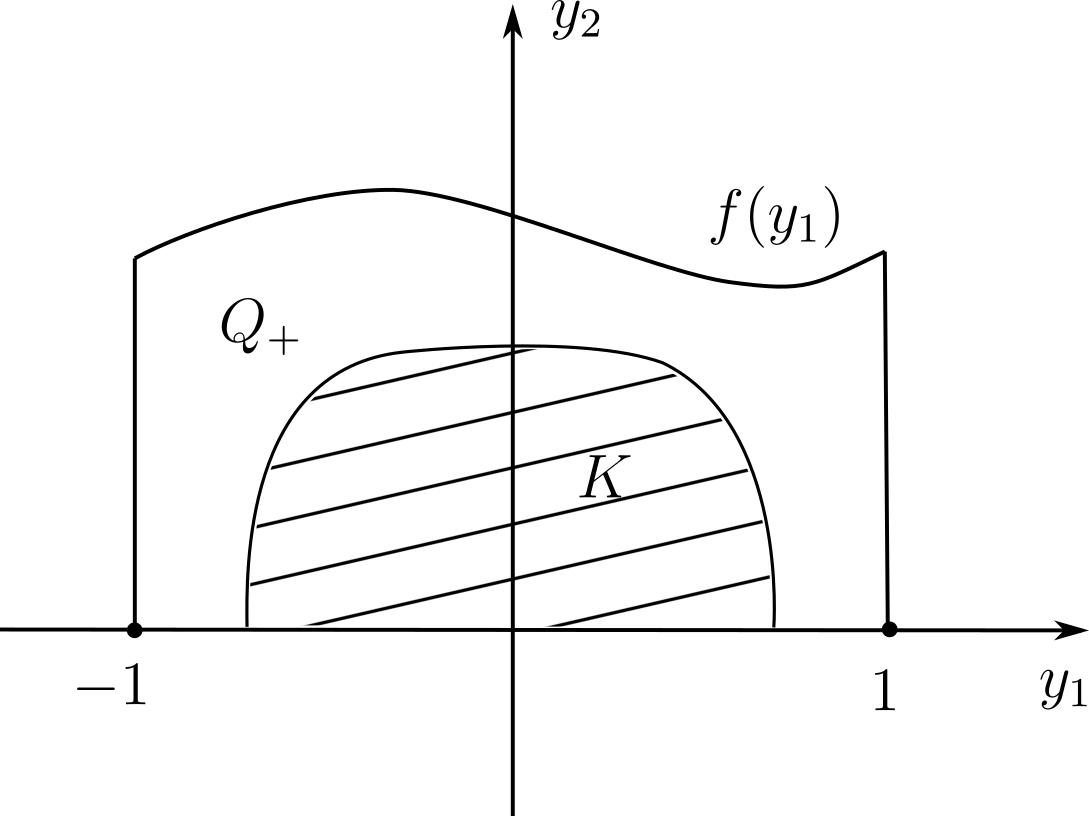}
\caption{The set $Q_+$ and $K$ in Lemma 3.1 }\label{fig:intro1}
\end{figure}

\begin{proof}
Denote $\tilde g$ as an extension of $g$ to $\mathbb R^2$ such that $\tilde g$ has
compact support and $\|\tilde g \|_{H^s(\mathbb R^2)} \lesssim \| g
\|_{H^s(Q_+)}$.  It is rather easy to construct such an extension by using reflection and
smooth truncation.  Clearly 
\begin{align}
\text{LHS of \eqref{eq2.1}} & \lesssim \int_{y\in Q_+} \frac {|g(y)|^2} { y_2^{2s}} dy \notag \\
&\lesssim \int_{|y_1|<1} \Bigl( \int_{0<y_2<1} |g(y_1,y_2)|^2 dy_2
+ \int_{0<y_2,\tilde y_2<1} \frac { |g(y_1,y_2) -g(y_1,\tilde y_2) |^2}
{ |y_2 -\tilde y_2|^{1+2s} } dy_2 d\tilde y_2 \Bigr) dy_1 \notag \\
&\lesssim \| \tilde g \|^2_{L_{y_1}^2 H^s_{y_2} (\mathbb R^2)}  \lesssim \| \tilde g \|^2_{H^s(\mathbb R^2)}
\lesssim \|g \|_{H^s(Q_+)}^2.
\end{align}
\end{proof}
\begin{lem} \label{lem2.2A}
Let $0<s<1/2$. Suppose $K$ is a nonempty compact set which is  properly contained in 
$Q_+=\{ (y_1, y_2):\, |y_1|<1, 0\le y_2 <f(y_1)\}$ where $f$ is a positive continuous function.
Assume $u$ vanishes on $Q_+\setminus K$. 
  Define $Q_-=\{(y_1,y_2): \, |y_1|<1, -f(y_1)<y_2<0\}$ and
\begin{align}
\tilde u(x_1,x_2)
=\begin{cases}
u(x_1,x_2), \qquad x \in Q_+; \\
-3 u(x_1,-x_2) +4 u(x_1,-\frac {x_2}2), \qquad x \in Q_-, \\
0, \qquad \text{otherwise}.
\end{cases}
\end{align}
Then
\begin{align}
\| \partial \tilde u \|_{H^s (\mathbb R^2)}  \le C_2 \| \partial u \|_{H^s (Q_+)},
\end{align}
where $C_2>0$ depends only on ($s$, $K$, $f$).
\end{lem}
\begin{proof}
Observe that
\begin{align}
\partial_1 \tilde u(x_1,x_2)
=\begin{cases}
\partial_1 u(x_1,x_2), \qquad x \in Q_+; \\
-3  \partial_1 u(x_1,-x_2) +4 \partial_1 u(x_1,-\frac {x_2}2), \qquad x \in Q_-;
\end{cases} \\
\partial_2 \tilde u(x_1,x_2)
=\begin{cases}
\partial_2 u(x_1,x_2), \qquad x \in Q_+; \\
3 (\partial_2 u)(x_1,-x_2) - 2 (\partial_2 u)(x_1,-\frac {x_2}2), \qquad x \in Q_-.
\end{cases}
\end{align}
By Lemma \ref{lem2.1A}, we have
\begin{align}
\| \partial \tilde u \|_{H^s (\mathbb R^2)}^2
& \lesssim \| \partial u \|_{H^s(Q_+)}^2
+ \| \partial \tilde u \|^2_{H^s(Q_-)}
+ \int_{x \in Q_+, y \in Q_-} 
\frac { |(\partial \tilde u)(x) - (\partial \tilde u)(y) |^2}
{|x-y|^{2+2s}} dx dy \notag \\
& \lesssim
\| \partial u \|_{H^s(Q_+)}^2 
+ \int_{x\in Q_+, y\in Q_+}
\frac { |(\partial u )(y_1, y_2) - (\partial u )(y_1, \frac {y_2}2) |^2}
{(|x_1-y_1|+|x_2+y_2|)^{2+2s}} dy dx  \notag \\
& \lesssim \| \partial u \|_{H^s(Q_+)}^2.
\end{align}
\end{proof}

\begin{lem} \label{Lem2.3A}
Let $0<s<\frac 12$ and $r_0>0$. Suppose $\gamma: \, \mathbb R \to \mathbb R$ is a 
$C^2$-function such that $\gamma(0)=0$, and  the region 
$\Omega=\{(x_1,x_2):\, |x|<r_0, x_2> \gamma(x_1) \} $ can be exactly 
prescribed by $\Omega= \{(x_1, x_2):
x_-<x_1<x_+,   \gamma(x_1)<x_2<\sqrt{r_0^2-x_1^2} \}$ for some $-r_0<x_-<x_+<r_0$. 
Assume $u$ is compactly supported in  $F_0= \{(x_1, x_2):
x_-+\delta_0<x_1<x_+-\delta_0,   \gamma(x_1)\le x_2<\sqrt{r_0^2-x_1^2} -\delta_0\}$ for some small $\delta_0>0$. Then there exists
an extension $\tilde u$ of $u$ which is compactly supported in $B(0,r_0)$, such that
\begin{align}
\| \tilde u \|_{L^2(\mathbb R^2)} 
+ \| \partial \tilde u \|_{H^s(\mathbb R^2)}
\le C_3  \cdot (\| u \|_{L^2(\Omega)} + \| \partial u \|_{H^s(\Omega)} ),
\end{align}
where $C_3>0$ depends on ($\delta_0$, $\gamma$, $s$).
\end{lem}
\begin{figure}[!h]
\centering
\includegraphics[width=0.4\textwidth]{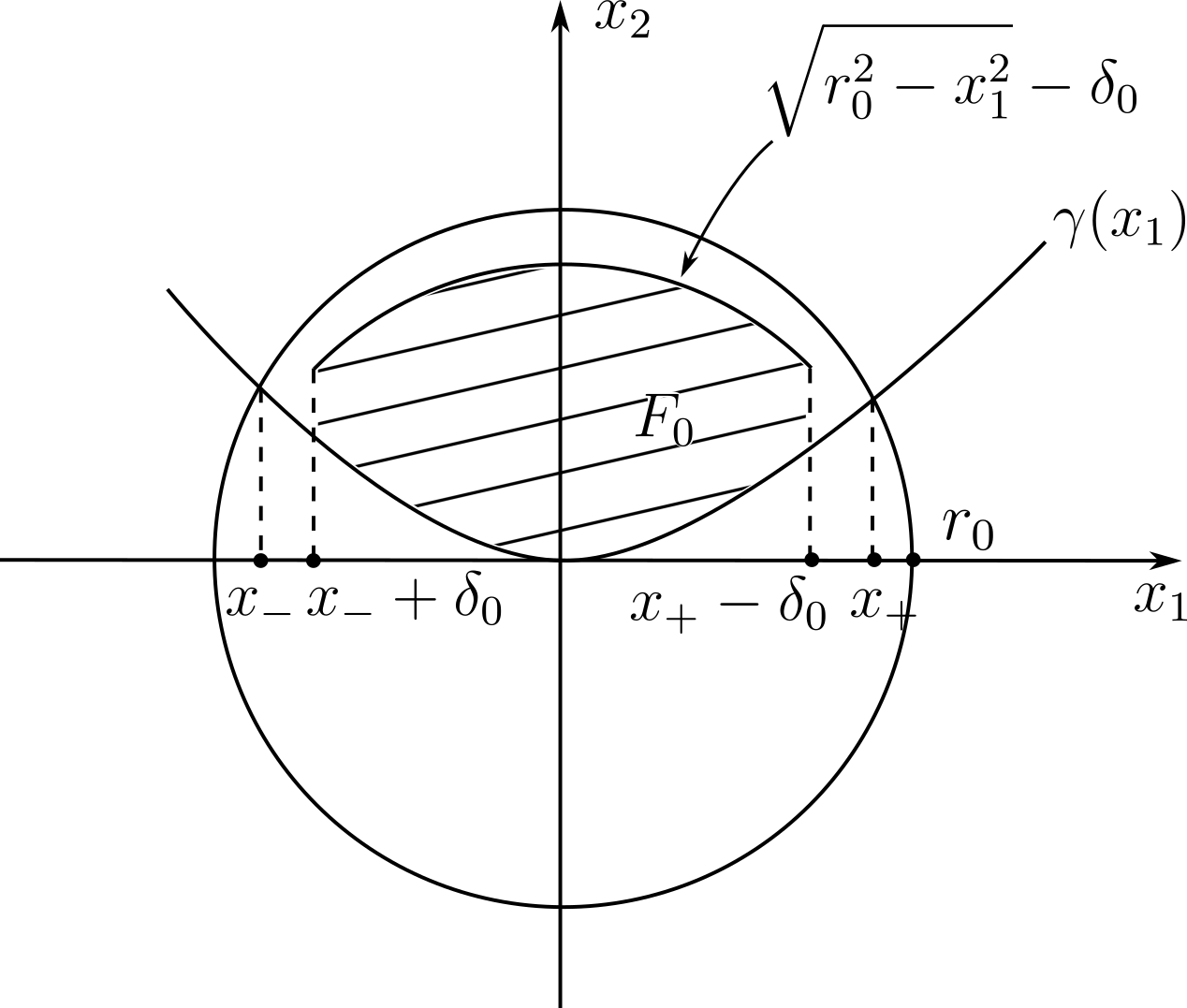}
\caption{The set $F_0$ and the curve $\gamma$ in Lemma 3.3 }\label{fig:intro1}
\end{figure}
\begin{proof}
Define the usual boundary straightening map by
$y = \psi (x)$ by $y_1=x_1$, $y_2=x_2-\gamma(x_1)$. The inverse map
is $x=\phi(y): \, x_1=y_1, x_2=y_2+\gamma(y_1)$. Define $v(y)= u(\phi(y))$,
for $y \in W= \psi(\Omega)$. 
Note that $(\partial v )(y) = (\partial u)(\phi(y)) (\partial \phi)(y)$. 
Since $\gamma \in C^2$, it is not 
difficult to check that
\begin{align}
\int_{y, \tilde y \in W}
\frac{ |(\partial v)(y) - (\partial v)(\tilde y) |^2}{|y-\tilde y|^{2+2s}}
dy d\tilde y
\lesssim \int_{x,\tilde x \in \Omega}
\frac{ |(\partial u)(x)- (\partial u)(\tilde x)|^2}
{|x-\tilde x|^{2+2s}} dx d\tilde x
+ \| \partial u \|_{L^2(\Omega)}^2,
\end{align}
where the second term on the RHS arises from the difference $(\partial \phi)(y) -(\partial \phi)(\tilde y)$. Denote $\Omega_0= \{(x_1, x_2):
x_-<x_1<x_+,   \gamma(x_1)\le x_2<\sqrt{r_0^2-x_1^2} \}$, 
$W_0=\psi(\Omega_0)$, and $K_1=\psi(F_0)$.
Note that $K_1$ is properly contained in $W_0$. By a minor adjustment of constants,
we can then apply Lemma \ref{lem2.2A} to obtain an extension $v$ to the whole $\mathbb R^2$. 
The map $\psi$ then provides the needed extension of $u$ which is compactly supported in
$B(0,r_0)$.  
\end{proof}

\begin{thm}[Extension of $H^{\frac 32-}$ functions] \label{thm3.1A}
Let $d\ge 1$ and $0<s<\frac 12$. Assume $\Omega \subset \mathbb R^d$ is a bounded domain with $C^2$ boundary.  Select a bounded open set $V$ such that $\Omega \subset \subset V$.  Then
there exists a bounded linear operator
\begin{align}
E: \, H^{1+s}(\Omega) \to H^{1+s} (\mathbb R^d)
\end{align}
such that for each $ f \in H^{1+s}(\mathbb R^d)$, we have
\begin{enumerate}
\item $Ef = f $ a.e. in $\Omega$;
\item $Ef$ has support within $V$;
\item Denote $\tilde f =Ef$, then 
\begin{align}
& \| \tilde f \|_{H^1(\mathbb R^d)} \lesssim \| f \|_{H^1(\Omega)}; \\
&\| \partial \tilde f \|_{ H^s(\mathbb R^d) }
\lesssim \| \partial f \|_{H^s( \Omega)}
= \| \partial f \|_{\dot H^s(\Omega)} + \| \partial f \|_{L^2(\Omega)}.
\end{align}
\end{enumerate}
\end{thm}
\begin{proof}
With no loss we consider the two dimensional case.    The argument then follows from 
a standard partition of unity.  The extension for the interior piece is quite straightforward.
The localized boundary piece follows from (after rotation and relabelling coordinate
axes if necessary)  Lemma \ref{Lem2.3A}.
\end{proof}

\begin{cor}[Boundedness of $T_2$ on a finite domain] \label{cor3.1A}
Let $d\ge 1$ and $0<s<\frac 12$. Assume $\Omega \subset \mathbb R^d$ is a bounded domain with $C^2$
boundary. Then
\begin{align}
&\| \partial f 1_{f>0} \|_{ H^s(\Omega)} \lesssim \|\partial f \|_{H^s(\Omega)}
=
\| \partial f \|_{\dot H^s(\Omega)} + \| \partial f \|_{L^2(\Omega)}; \\
&\| T_2 f\|_{H^{1+s}(\Omega)} \lesssim \| f \|_{H^{1+s}(\Omega)}.
\end{align}
\end{cor}
\begin{proof}
By Theorem \ref{thm3.1A}, we extend $f$ to $\tilde f$ defined on the whole $\mathbb R^d$.
The result then follows from the boundedness in the whole space case (see Corollary \ref{cor2.0A}).
\end{proof}
Finally we remark that Theorem \ref{thmMa} follows from Corollary \ref{cor3.1A} since the proofs
for $T_1$ and $T_3$ are similar.

\bibliographystyle{abbrv}

\begin{thebibliography}{}

\bibitem{BS11}
Bourdaud, G., Sickel, W. Composition operators on function spaces with fractional order of smoothness. In: Ozawa, T., Sugimoto, M. (eds.) Harmonic Analysis and Nonlinear Partial 
Differential Equations, pp. 93--132 RIMS Kokyuroku Bessatsu, B26, Res. Inst. Math. Sci. 
(RIMS), Kyoto (2011)


\bibitem{BD11}
Bogdan, K. and Dyda, B.
The best constant in a fractional Hardy inequality. 
Math. Nachr. 284 (2011), no. 5-6, 629--638.

\bibitem{Dyda04}
Dyda, B. A fractional order Hardy inequality.
Illinois J. of Math., 48, 575--588 (2004)


\bibitem{LM72}
Lions J.-L., and Magenes E. 
Non Homogeneous boundary value problems and applications I.
Springer, Berlin (1972)




\bibitem{LRNew}
Li, D. and  Rodrigo, J.
Remarks on a nonlocal transport. 
Adv. Math. 374 (2020), 107345, 27 pp.

\bibitem{Mey91}
Bourdaud, G., Meyer, Y. Fonctions qui op\'erent sur les espaces de Sobolev. J. Funct. Anal. 97(2), 
351--360 (1991).


\bibitem{MN17}
Musina, R., Nazarov, A.I. A note on truncations in fractional Sobolev spaces. Bull. Math. Sci. (2017). https://doi.org/10.1007/s13373-017-0107-8



\bibitem{Os92}
Oswald, P. On the boundedness of the mapping $f\to |f|$ in Besov spaces. Comment. Math. Univ. Carolin. 33(1), 57--66 (1992)


\bibitem{S95}
Savar\'e, G. On the regularity of the positive part of functions.
Non. Ana., Theo. Meth. \& Appl. Vol. 27, No. 9, pp, 1055--1074 (1996)



\end{thebibliography}


\end{document}